[1994/12/01]
\documentclass{ijmart-mod}
\chardef\bslash=`\\ 

\usepackage{graphicx}
\usepackage[breaklinks=true]{hyperref}
\usepackage{mathtools}
\usepackage{caption}
\usepackage{amsmath}
\usepackage{array}
\usepackage{multirow}
\usepackage{soul}

\newtheorem{thm}{Theorem}[section]

\newtheorem{lem}[thm]{Lemma}
\newtheorem{prop}[thm]{Proposition}

\theoremstyle{definition}

\newtheorem{rem}[thm]{Remark}

\theoremstyle{remark}

\newcommand{\eval}[2][\right]{\relax
  \ifx#1\right\relax \left.\fi#2#1\rvert}

\begin{document}
\title{Small kissing polytopes}

\author[A. Deza]{Antoine Deza}
\address{McMaster University, Hamilton, Ontario, Canada}
\email{deza@mcmaster.ca} 

\author[Z. Liu]{Zhongyuan Liu}
\address{McMaster University, Hamilton, Ontario, Canada}
\email{liu164@mcmaster.ca} 

\author[L. Pournin]{Lionel Pournin}
\address{Universit{\'e} Paris 13, Villetaneuse, France}
\email{lionel.pournin@univ-paris13.fr}

\maketitle

{\centering \large \emph{Dedicated to Tam{\'a}s Terlaky on the occasion of his 70th birthday}\par}

\begin{abstract}
A lattice $(d,k)$-polytope is the convex hull of a set of points in $\mathbb{R}^d$ whose coordinates are integers ranging between $0$ and $k$. We consider the smallest possible distance $\varepsilon(d,k)$ between two disjoint lattice $(d,k)$-polytopes. 
We propose an algebraic model for this distance and derive from it an explicit formula for $\varepsilon(2,k)$. Our model also allows for the computation of previously intractable values of $\varepsilon(d,k)$. In particular, we compute $\varepsilon(3,k)$ when $4\leq{k}\leq8$, $\varepsilon(4,k)$ when $2\leq{k}\leq3$, and  $\varepsilon(6,1)$.\\
\end{abstract}


\section{Introduction}\label{DLP.sec.1}
The smallest possible distance between two disjoint lattice $(d,k)$-polytopes---convex hulls of sets of points with integer coordinates in $[0,k]^d$---is a natural quantity in discrete geometry. This quantity, which we refer to as $\varepsilon(d,k)$ in the sequel, is connected to the complexity of algorithms such as the linear minimization formulation by G{\'a}bor Braun, Sebastian Pokutta, and Robert Weismantel~\cite{BraunPokuttaWeismantel2022} of the von Neumann alternating projections algorithm~\cite{VonNeumann1949}. It is also related to several notions that appear in optimization. For instance, the \emph{facial distance} of a polytope $P$, studied by Javier Pe{\~n}a and Daniel Rodriguez~\cite{PenaRodriguez2019} and by David Gutman and Javier Pe{\~n}a~\cite{GutmanPena2018, Pena2019}, is the smallest possible distance between a face $F$ of $P$ and the convex hull of the vertices of $P$ that are not contained in $F$. The \emph{vertex-facet distance} of a polytope $P$, considered by Amir Beck and Shimrit Shtern \cite{BeckShtern2017}, is the smallest possible distance between the affine hull of a facet $F$ of $P$ and a vertex of $P$ that does not belong to $F$. The smallest possible vertex-facet distance of a lattice $(d,1)$-simplex has been estimated by Noga Alon and V\u{a}n V{\~u} \cite{AlonVU1997}. Another such notion is the \emph{pyramidal width} of a finite set of points, investigated by Simon Lacoste-Julien and Martin Jaggi~\cite{Lacoste-JulienJaggi2015} and by Luis Rademacher and Chang Shu~\cite{RademacherShu2022}, which coincides with the facial distance of the convex hull of these points~\cite{PenaRodriguez2019}. G{\'a}bor Braun, Alejandro Carderera, Cyrille Combettes, Hamed Hassani, Amin Karbasi, Aryan Mokhtari, and Sebastian Pokutta provide a comprehensive overview of these notions in~\cite{BraunCardereraCombettesHassaniKarbasiMokhtariPokutta2022}.

Lower and upper bounds on $\varepsilon(d,k)$ that are almost matching as $d$ goes to infinity and a number of properties of this quantity as a function of $d$ and $k$ have been established by Shmuel Onn, Sebastian Pokutta, and two of the authors in~\cite{DezaOnnPokuttaPournin2024}. The values of $\varepsilon(2,k)$ when $1\leq{k}\leq6$, of $\varepsilon(3,k)$ when $1\leq{k}\leq3$, of $\varepsilon(4,1)$, and of $\varepsilon(5,1)$ have been computed as a consequence of these properties and are reported in~\cite{DezaOnnPokuttaPournin2024} (see also the non-bolded entries in Table~\ref{DLP.sec.2.tab.2}).

An algebraic model that allows for the computation of previously intractable values of $\varepsilon(d,k)$ is developed in Section~\ref{DLP.sec.2}. More precisely, $\varepsilon(d,k)$ is bounded by the smallest non-zero value of a certain algebraic fraction over a subset of the lattice points contained in the hypercube $[-k,k]^{d^2}$. Using this model, we provide the following formula for $\varepsilon(2,k)$ in Section~\ref{DLP.sec.3}.
\begin{thm}\label{DLP.sec.3.thm.1}
If $k$ is greater than $1$, then
$$
\varepsilon(2,k)=\frac{1}{\sqrt{(k-1)^2+k^2}}\mbox{.}
$$
\end{thm}

\begin{table}
\begin{center}
\begin{tabular}{>{\centering}p{0.5cm}cccccccc}
\multirow{2}{*}{$d$}&  \multicolumn{8}{c}{$k$}\\
\cline{2-9}
 & $1$ & $2$ & $3$ & $4$ & $5$ & $6$ & $7$ & $8$\\
\hline & \\[-1.1\bigskipamount]
$3$ & $\sqrt{6}$ & \!$5\sqrt{2}$\! & \!$\sqrt{299}$\! & \!$\bf{5\sqrt{42}}$\! & \!$\bf{\sqrt{2870}}$\! & \!$\bf{\sqrt{6466}}$\! & \!$\bf{5\sqrt{510}}$\! & \!$\bf{\sqrt{22826}}$\!\\
$4$ & $3\sqrt{2}$ & \!$\bf{2\sqrt{113}}$\! & \!$\bf{11\sqrt{71}}$\!\\
$5$ & $\sqrt{58}$\\
$6$ & $\bf{\sqrt{202}}$ &
\end{tabular}
\end{center}
\caption{The known values of $1/\varepsilon(d,k)$ when $d$ is at least $3$.}\label{DLP.sec.2.tab.2}
\end{table}

Finally, we show in Section~\ref{DLP.sec.4} how the subset of the lattice points in $[-k,k]^{d^2}$ over which the minimization is performed can be reduced, and discuss the computational efficiency of the resulting strategy. This makes it possible to determine values of $\varepsilon(d,k)$ whose computation was previously intractable. Using this strategy, we compute $\varepsilon(3,k)$ when $4\leq{k}\leq8$, $\varepsilon(4,k)$ when $2\leq{k}\leq3$, and  $\varepsilon(6,1)$. These values of $\varepsilon(d,k)$ are the inverse of the numbers shown in bold in Table~\ref{DLP.sec.2.tab.2}. In addition, for each of the obtained values of $\varepsilon(d,k)$, we provide an explicit pair of lattice $(d,k)$-polytopes whose distance is precisely $\varepsilon(d,k)$. We shall refer to such a pair of polytopes as \emph{kissing polytopes}.

\section{A least squares model for polytope distance}\label{DLP.sec.2}

Let us consider two disjoint lattice $(d,k)$-simplices $P$ and $Q$ whose affine hulls are disjoint. Denote by $p^0$ to $p^{n}$ the vertices of $P$ and by $q^0$ to $q^{m}$ the vertices of $Q$ where $n$ and $m$ denote the dimension of $P$ and $Q$, respectively. The distance of $P$ and $Q$ is the smallest possible value of
\begin{equation}\label{DLP.sec.1.eq.1}
\left\|\sum_{i=0}^n\lambda_ip^i-\sum_{i=0}^m\mu_iq^i\right\|^2
\end{equation}
where $\lambda_0$ to $\lambda_n$ and $\mu_0$ to $\mu_m$ are two sets of non-negative numbers that each sum to $1$. The constraint that each of these sets of numbers sum to $1$ can be replaced by expressing $\lambda_0$ and $\mu_0$ as
\begin{equation}\label{DLP.sec.1.eq.1.25}
\left\{
\begin{array}{l}
\displaystyle\lambda_0=1-\sum_{i=1}^n\lambda_i\\[1.5\bigskipamount]
\displaystyle\mu_0=1-\sum_{i=1}^m\mu_i
\end{array}
\right.
\end{equation}
and substituting them in (\ref{DLP.sec.1.eq.1}) by these expressions. As a consequence,
$$
d(P,Q)^2=\min_{\substack{\lambda\in\Delta_n\\\mu\in\Delta_m}}f_{P,Q}(\lambda,\mu)
$$
where
\begin{equation}\label{DLP.sec.1.eq.1.5}
f_{P,Q}(\lambda,\mu)=\left\|p^0-q^0+\sum_{i=1}^n\lambda_i\bigl(p^i-p^0\bigr)-\sum_{i=1}^m\mu_i\bigl(q^i-q^0\bigr)\right\|^2
\end{equation}
and $\Delta_j$ denotes the $j$-dimensional simplex
$$
\Delta_j=\biggl\{x\in[0,+\infty[^j:\sum_{i=1}^jx_i\leq1\biggr\}\mbox{.}
$$

We will consider $f_{P,Q}$ as a function from $\mathbb{R}^n\!\mathord{\times}\mathbb{R}^m$ to $[0,+\infty[$. Note that this function depends on the ordering of the vertices of $P$ and $Q$ but this ordering will not play a role in the sequel, and we assume that a prescribed ordering has been fixed for the vertices of each pair of polytopes $P$ and $Q$. 

Relaxing the constraint that $\lambda$ and $\mu$ should be contained in $\Delta_n$ and $\Delta_m$ provides a lower bound on $d(P,Q)$ of the form
\begin{equation}\label{DLP.sec.1.eq.2}
d(P,Q)^2\geq\min_{\substack{\lambda\in\mathbb{R}^n\\\mu\in\mathbb{R}^m}}f_{P,Q}(\lambda,\mu)\mbox{.}
\end{equation}

Note that the right-hand side of (\ref{DLP.sec.1.eq.2}) is the distance between the affine hull of $P$ and the affine hull of $Q$. In particular, the accuracy of this bound is related to how close the distance of $P$ and $Q$ is to the distance of their affine hulls.

Now consider the $d\mathord{\times}(d-1)$ matrix
\begin{equation}\label{DLP.sec.1.eq.4}
A=\left[
\begin{array}{cccccc}
p^1_1-p^0_1 & \cdots & p^n_1-p^0_1 & q^1_1-q^0_1 & \cdots & q^m_1-q^0_1\\
\vdots & & \vdots & \vdots & & \vdots\\
p^1_d-p^0_d & \cdots & p^n_d-p^0_d & q^1_d-q^0_d & \cdots & q^m_d-q^0_d\\
\end{array}
\right]
\end{equation}
and the vector
\begin{equation}\label{DLP.sec.1.eq.4.5}
b=q^0-p^0\mbox{.}
\end{equation}

It will be important to keep in mind that $A$ and $b$ depend on $P$ and $Q$. Observe that, with these notations, (\ref{DLP.sec.1.eq.1.5}) can be rewritten into
\begin{equation}\label{DLP.sec.1.eq.4.6}
f_{P,Q}(\lambda,\mu)=\left\|A\chi-b\right\|^2
\end{equation}
where $\chi$ is the vector such that
\begin{equation}\label{DLP.sec.1.eq.4.25}
\chi^t=\bigl[\lambda_1, \cdots,\lambda_n,-\mu_1,\cdots,-\mu_m\bigr]\mbox{.}
\end{equation}

\begin{rem}
According to (\ref{DLP.sec.1.eq.4.6}), $f_{P,Q}(\lambda,\mu)$ is the sum of the squares of the coordinates of $A\chi-b$. In particular, negating both a row of $A$ and the corresponding coefficient of $b$ will not change the value of that function either. Likewise, negating a column of $A$ and the corresponding row of $\chi$ will not change the value of $f_{P,Q}(\lambda,\mu)$. As a consequence, the right-hand side of (\ref{DLP.sec.1.eq.2}) does not change when a subset of the columns of $A$ are negated or a subset of its rows are negated together with the corresponding coefficients of $b$.
\end{rem}

Let us now give an expression for the right-hand side of (\ref{DLP.sec.1.eq.2}).

\begin{lem}\label{DLP.sec.1.lem.1}
$f_{P,Q}$ admits a unique minimum over $\mathbb{R}^n\!\mathord{\times}\mathbb{R}^m$ if and only if $A^tA$ is non-singular. Moreover, in that case,
\begin{equation}\label{DLP.sec.1.lem.1.eq.1}
\min_{\substack{\lambda\in\mathbb{R}^n\\\mu\in\mathbb{R}^m}}f_{P,Q}(\lambda,\mu)=\bigl\|A(A^tA)^{-1}A^tb-b\bigr\|^2\mbox{.}
\end{equation}
\end{lem}
\begin{proof}
The minimum of $f_{P,Q}$ is reached at a pair $(\lambda,\mu)$ from $\mathbb{R}^n\!\mathord{\times}\mathbb{R}^m$ such that all the partial derivatives of $f_{P,Q}$ simultaneously vanish, that is when
$$
\frac{\partial{f_{P,Q}}}{\partial{\lambda_i}}(\lambda,\mu)=0
$$
for all $i$ satisfying $1\leq{i}\leq{n}$ and
$$
\frac{\partial{f_{P,Q}}}{\partial{\mu_i}}(\lambda,\mu)=0
$$
for all $i$ satisfying $1\leq{i}\leq{m}$. Since $f_{P,Q}$ is a quadratic function of $\lambda$ and $\mu$, its partial derivatives are linear. In other words, finding the minimum of $f_{P,Q}$ over  $\mathbb{R}^n\!\mathord{\times}\mathbb{R}^m$ amounts to solve a least squares problem. In particular setting to $0$ all of these partial derivatives results in the system of linear equalities
\begin{equation}\label{DLP.sec.1.eq.3}
A^tA\chi=A^tb\mbox{.}
\end{equation}

Since $f_{P,Q}$ is a convex quadratic function, the solutions of (\ref{DLP.sec.1.eq.3}) correspond bijectively via (\ref{DLP.sec.1.eq.4.25}) with the pairs $(\lambda,\mu)$ such that $f_{P,Q}$ is minimal. It immediately follows that $f_{P,Q}$ admits a unique minimum over $\mathbb{R}^n\!\mathord{\times}\mathbb{R}^m$ if and only if $A^tA$ is non-singular. Moreover, in that case, the unique solution of (\ref{DLP.sec.1.eq.3}) is
$$
\chi=(A^tA)^{-1}A^tb
$$
and substituting this expression of $\chi$ in (\ref{DLP.sec.1.eq.4.6}) completes the proof.
\end{proof}

According to (\ref{DLP.sec.1.eq.2}), Lemma~\ref{DLP.sec.1.lem.1} provides a lower bound on $d(P,Q)$ in the case when $A^tA$ is non-singular. The following remark provides a necessary and sufficient condition on $(A^tA)^{-1}A^tb$ for this bound to be sharp.

\begin{rem}\label{DLP.sec.1.rem.1}
Recall that the minimum of $f_{P,Q}$ over $\mathbb{R}^n\!\mathord{\times}\mathbb{R}^m$ is the distance between the affine hulls of $P$ and $Q$. Therefore, if $A^tA$ is non-singular, then according to (\ref{DLP.sec.1.eq.4.6}), (\ref{DLP.sec.1.eq.4.25}), and Lemma~\ref{DLP.sec.1.lem.1}, the first $n$ coordinates of
$$
\chi=(A^tA)^{-1}A^tb
$$
provide an affine combination $p^\star$ of the vertices of $P$ and its last $m$ coordinates an affine combination $q^\star$ of the vertices of $Q$ such that
$$
d(p^\star,q^\star)=d\bigl(\mathrm{aff}(P),\mathrm{aff}(Q)\bigr)\mbox{.}
$$

In particular, if the first $n$ coefficients of $\chi$ are all non-negative and sum to at most $1$ while its last $m$ coefficients are all non-positive and sum to at least $-1$, then $p^\star$ belongs to $P$ and $q^\star$ to $Q$. In that case, the distance of $P$ and $Q$ coincides with the distance of their affine hulls. Otherwise the distance of $P$ and $Q$ is strictly greater than the distance of their affine hulls.
\end{rem}

We shall now focus on certain pairs of simplices whose distance is precisely $\varepsilon(d,k)$. The following is proven in \cite{DezaOnnPokuttaPournin2024} (see Theorem 5.2 therein).

\begin{thm}\label{DHLOPP.sec.4.lem.1}
There exist two lattice $(d,k)$-polytopes $P$ and $Q$ such that
\begin{enumerate}
\item[(i)] $d(P,Q)$ is equal to $\varepsilon(d,k)$,
\item[(ii)] both $P$ and $Q$ are simplices,
\item[(iii)] $\mathrm{dim}(P)+\mathrm{dim}(Q)$ is equal to $d-1$, and
\item[(iv)] the affine hulls of $P$ and $Q$ are disjoint.
\end{enumerate}
\end{thm}

We shall prove that when $P$ and $Q$ satisfy the assertions (i) to (iv) in the statement of Theorem~\ref{DHLOPP.sec.4.lem.1}, $f_{P,Q}$ admits a unique minimum over $\mathbb{R}^n\!\mathord{\times}\mathbb{R}^m$ as a consequence of two results from \cite{DezaOnnPokuttaPournin2024}. The first of these results states that
\begin{equation}\label{DLP.sec.1.eq.2.5}
d(P,Q)\geq\varepsilon\bigl(\mathrm{dim}(P\cup{Q}),k\bigr)
\end{equation}
(see Lemma 4.3 in \cite{DezaOnnPokuttaPournin2024}) and the second that, when $k$ is fixed, $\varepsilon(d,k)$ is a strictly decreasing function of $d$ (see Theorem 5.1 in \cite{DezaOnnPokuttaPournin2024}).
%

\begin{prop}\label{DLP.sec.1.thm.1}
If $P$ and $Q$ satisfy the assertions (i) to (iv) in the statement of Theorem~\ref{DHLOPP.sec.4.lem.1}, then $f_{P,Q}$ has a unique minimum over $\mathbb{R}^n\!\mathord{\times}\mathbb{R}^m$.
\end{prop}
\begin{proof}
Denote by $\mathrm{aff}(P)$ and $\mathrm{aff}(Q)$ the affine hulls of $P$ and $Q$, respectively. Consider a point $p^\star$ in $\mathrm{aff}(P)$ and a point $q^\star$ in $\mathrm{aff}(Q)$ such that
\begin{equation}\label{DLP.sec.1.thm.1.eq.1}
\bigl\|q^\star-p^\star\bigr\|=d\bigl(\mathrm{aff}(P),\mathrm{aff}(Q)\bigr)\mbox{.}
\end{equation}

By the above discussion, the pair $(p^\star,q^\star)$ corresponds to a point $(\lambda^\star,\mu^\star)$ in $\mathbb{R}^n\!\mathord{\times}\mathbb{R}^m$ where $f_{P,Q}$ reaches its minimum. Assume that $f_{P,Q}$ also reaches its minimum at a point $(\overline{\lambda}^\star,\overline{\mu}^\star)$ in $\mathbb{R}^n\!\mathord{\times}\mathbb{R}^m$ different from $(\lambda^\star,\mu^\star)$. By construction, $(\overline{\lambda}^\star,\overline{\mu}^\star)$ then provides the coefficients of an affine combination $\overline{p}^\star$ of $p^0$ to $p^n$ and of an affine combination $\overline{q}^\star$ of $q^0$ to $q^m$ such that
\begin{equation}\label{DLP.sec.1.thm.1.eq.2}
\bigl\|\overline{q}^\star-\overline{p}^\star\bigr\|=d\bigl(\mathrm{aff}(P),\mathrm{aff}(Q)\bigr)\mbox{.}
\end{equation}

Since $p^\star$ and $\overline{p}^\star$ are both contained in $\mathrm{aff}(P)$, so is their midpoint. Likewise, the midpoint of $q^\star$ and $\overline{q}^\star$ is contained in $\mathrm{aff}(Q)$. Hence,
\begin{equation}\label{DLP.sec.1.thm.1.eq.3}
d\bigl(\mathrm{aff}(P),\mathrm{aff}(Q)\bigr)\leq\biggl\|\frac{q^\star+\overline{q}^\star}{2}-\frac{p^\star+\overline{p}^\star}{2}\biggl\|\mbox{.}
\end{equation}

However, by the triangle inequality,
\begin{equation}\label{DLP.sec.1.thm.1.eq.4}
\biggl\|\frac{q^\star+\overline{q}^\star}{2}-\frac{p^\star+\overline{p}^\star}{2}\biggl\|\leq\frac{1}{2}\bigl\|q^\star-p^\star\bigr\|+\frac{1}{2}\bigl\|\overline{q}^\star-\overline{p}^\star\bigr\|
\end{equation}
with equality if and only if $q^\star-p^\star$ is a multiple of $\overline{q}^\star-\overline{p}^\star$ by a positive coefficient or one of these vectors is equal to $0$. Under the assumption that assertion (iv) from the statement of Theorem~\ref{DHLOPP.sec.4.lem.1} holds, these vectors are both non-zero. Hence, $q^\star-p^\star$ is a multiple of $\overline{q}^\star-\overline{p}^\star$ by a positive coefficient and according to (\ref{DLP.sec.1.thm.1.eq.1}) and (\ref{DLP.sec.1.thm.1.eq.2}), these vectors must therefore be equal. It immediately follows that the vectors $\overline{p}^\star-p^\star$ and $\overline{q}^\star-q^\star$ also coincide.

Now recall that $(\lambda^\star,\mu^\star)$ and $(\overline{\lambda}^\star,\overline{\mu}^\star)$ are different. As a consequence, so are the pairs $(p^\star,q^\star)$ and $(\overline{p}^\star,\overline{q}^\star)$. Since the vectors $\overline{p}^\star-p^\star$ and $\overline{q}^\star-q^\star$ coincide they must therefore be non-zero. Hence, the translates of $\mathrm{aff}(P)$ and of $\mathrm{aff}(Q)$ through the origin of $\mathbb{R}^d$ intersect in a non-zero vector and the dimension of their union must be less than $\mathrm{dim}(P)+\mathrm{dim}(Q)$. As a consequence
$$
\mathrm{dim}(P\cup{Q})\leq\mathrm{dim}(P)+\mathrm{dim}(Q)\mbox{.}
$$

Therefore, under the assumption that $P$ and $Q$ satisfy the assertion (iii) in the statement of Theorem~\ref{DHLOPP.sec.4.lem.1}, it follows that $P\cup{Q}$ has dimension at most $d-1$. According to (\ref{DLP.sec.1.eq.2.5}), this implies that the distance between $P$ and $Q$ is at least $\varepsilon(d-1,k)$. As a consequence, if the assertion (i) in the statement of Theorem~\ref{DHLOPP.sec.4.lem.1} holds for $P$ and $Q$, then one obtains that
$$
\varepsilon(d,k)\geq\varepsilon(d-1,k)\mbox{.}
$$

However, Theorem 5.1 in \cite{DezaOnnPokuttaPournin2024} states that $\varepsilon(d,k)<\varepsilon(d-1,k)$. By this contradiction, $f_{P,Q}$ has a unique minimum over $\mathbb{R}^n\!\mathord{\times}\mathbb{R}^m$.
\end{proof}

Combining Proposition~\ref{DLP.sec.1.thm.1}, Lemma~\ref{DLP.sec.1.lem.1}, and Theorem~\ref{DHLOPP.sec.4.lem.1}, one obtains a lower bound on $\varepsilon(d,k)$ from (\ref{DLP.sec.1.eq.2}) of the form
\begin{equation}\label{DLP.sec.1.eq.5}
\varepsilon(d,k)\geq\min_{P,Q}\Bigl\{\bigl\|A(A^tA)^{-1}A^tb-b\bigr\|\Bigr\}
\end{equation}
where the minimum ranges over the pairs of lattice $(d,k)$-simplices $P$ and $Q$ whose dimensions sum to $d-1$, for which the matrix $A$ obtained from (\ref{DLP.sec.1.eq.4}) is such that $A^tA$ is non-singular and the vector $b$ obtained from (\ref{DLP.sec.1.eq.4.5}) satisfies
$$
A(A^tA)^{-1}A^tb\neq{b}\mbox{.}
$$


\section{The $2$-dimensional case}\label{DLP.sec.3}

In this section, we give a formula for $\varepsilon(2,k)$ using the model described in Section~\ref{DLP.sec.1}. Consider two disjoint lattice $(2,k)$-polytopes $P$ and $Q$ that satisfy assertions (i) to (iv) from the statement of Theorem~\ref{DHLOPP.sec.4.lem.1}. Since the dimensions of $P$ and $Q$ sum to $1$, one of these polytopes has dimension $0$ and the other has dimension $1$. We assume without loss of generality that $P$ is a line segment and that $Q$ is made of a single point by exchanging them if needed.

Let us first observe that according to (\ref{DLP.sec.1.eq.4}) and (\ref{DLP.sec.1.eq.4.5}),
$$
\left\{
\begin{array}{l}
A=p^1-p^0\mbox{,}\\
b=q^0-p^0\mbox{.}\\
\end{array}
\right.
$$

As a consequence, (\ref{DLP.sec.1.eq.5}) simplifies into
$$
\varepsilon(2,k)\geq\frac{|(p^1_2-p^0_2)(q^0_1-p^0_1)-(p^1_1-p^0_1)(q^0_2-p^0_2)|}{\sqrt{(p^1_1-p^0_1)^2+(p^1_2-p^0_2)^2}}
$$

It follows that $\varepsilon(2,k)$ is at least the smallest possible value of
\begin{equation}\label{DLP.sec.3.eq.1}
\frac{|x_2x_3-x_1x_4|}{\sqrt{x_1^2+x_2^2}}
\end{equation}
over all the lattice points $x$ contained in the hypercube $[-k,k]^4$ such that neither $x_1^2+x_2^2$ nor $x_2x_3-x_1x_4$ is equal to $0$. We bound (\ref{DLP.sec.3.eq.1}) as follows.
\begin{lem}\label{DLP.sec.3.lem.1}
If $k$ is greater than $1$ then, for every lattice point $x$ in $[-k,k]^4$ such that $x_1^2+x_2^2$ and $x_2x_3-x_1x_4$ are both non-zero,
\begin{equation}\label{DLP.sec.3.lem.1.eq.0}
\frac{|x_2x_3-x_1x_4|}{\sqrt{x_1^2+x_2^2}}\geq\frac{1}{\sqrt{(k-1)^2+k^2}}\mbox{.}
\end{equation}
\end{lem}
\begin{proof}
Consider a lattice point $x$ in $[-k,k]^4$ such that neither $x_1^2+x_2^2$ nor $x_2x_3-x_1x_4$ is equal to $0$. We assume without loss of generality that $x_1$ and $x_2$ are non-negative thanks to the symmetries of $[-k,k]^4$. We review two cases.

First assume that $x_1$ and $x_2$ coincide. In that case,
\begin{equation}\label{DLP.sec.3.lem.1.eq.1}
\frac{|x_2x_3-x_1x_4|}{\sqrt{x_1^2+x_2^2}}=\frac{|x_3-x_4|}{\sqrt{2}}\mbox{.}
\end{equation}

Since $x_2x_3-x_1x_4$ is not equal to $0$, $x_3$ and $x_4$ cannot coincide and the right-hand side of (\ref{DLP.sec.3.lem.1.eq.1}) is at least $1/\sqrt{2}$. As $k$ is greater than $1$,
$$
\frac{1}{\sqrt{(k-1)^2+k^2}}\leq\frac{1}{\sqrt{2}}
$$
and the lemma follows in this case.

Now assume that $x_1$ and $x_2$ are different. Since $|x_2x_3-x_1x_4|$ is at least $1$,
\begin{equation}\label{DLP.sec.3.lem.1.eq.2}
\frac{|x_2x_3-x_1x_4|}{\sqrt{x_1^2+x_2^2}}\geq\frac{1}{\sqrt{x_1^2+x_2^2}}
\end{equation}

Recall that $x_1$ and $x_2$ are integers contained in $[0,k]$. Since they are different, one of them is at most $k-1$. As a consequence,
$$
\frac{1}{\sqrt{x_1^2+x_2^2}}\geq\frac{1}{\sqrt{(k-1)^2+k^2}}
$$
and combining this with (\ref{DLP.sec.3.lem.1.eq.2}) completes the proof.
\end{proof}

By Lemma~\ref{DLP.sec.3.lem.1}, and the preceding discussion, $\varepsilon(2,k)$ is at least the right-hand side of (\ref{DLP.sec.3.lem.1.eq.0}) when $k$ is at least $2$. Theorem~\ref{DLP.sec.3.thm.1} states that this is sharp.

\begin{proof}[Proof of Theorem~\ref{DLP.sec.3.thm.1}]
It suffices to exhibit a lattice point $P$ and a lattice segment $Q$ contained in the square $[0,k]^2$ such that
$$
d(P,Q)=\frac{1}{\sqrt{(k-1)^2+k^2}}\mbox{.}
$$

Such an example is provided by the point in $[0,k]^2$  whose two coordinates are equal to $1$ and any of the two line segments that are incident to the origin and whose other vertex has coordinates $k$ and $k-1$. This point and this line segment are represented in Figure~\ref{DLP.sec.3.fig.1} when $k$ is equal to $4$.
\end{proof}

\begin{figure}[t]
\begin{centering}
\includegraphics[scale=1]{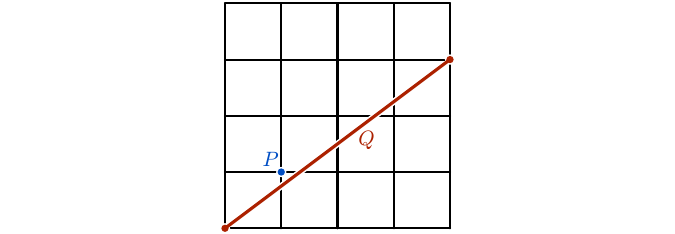}
\caption{A pair of kissing lattice $(2,4)$-polytopes.}\label{DLP.sec.3.fig.1}
\end{centering}
\end{figure}

\begin{rem}\label{DLP.sec.3.rem.1}
The strategy exposed in this section in the $2$-dimensional case can be generalized to any higher dimension. In particular, a quotient similar to (\ref{DLP.sec.3.eq.1}) can be explicitly computed in any fixed dimension $d$ that depends on a lattice point $x$ contained in the hypercube $[-k,k]^{d^2}$. The first $d(d-1)$ coordinates of $x$ are the entries of $A$ and its last $d$ coordinates are the coordinates of $b$. The minimum of that quotient under the constraint that its numerator and denominator are positive provides a lower bound on $\varepsilon(d,k)$. For instance, when $d$ is equal to $3$, the minimal value of the ratio
\begin{equation}\label{DLP.sec.3.rem.1.eq.1}
\displaystyle\frac{|x_1(x_6x_8-x_5x_9)+x_2(x_4x_9-x_6x_7)+x_3(x_5x_7-x_4x_8)|}{\sqrt{(x_1x_5-x_2x_4)^2+(x_1x_6-x_3x_4)^2+(x_2x_6-x_3x_5)^2}}
\end{equation}
over all the lattice points $x$ in $[-k,k]^9$ such that the numerator and the denominator of (\ref{DLP.sec.3.rem.1.eq.1}) are positive is a lower bound on $\varepsilon(3,k)$. However, the expression for this quotient gets complicated as the dimension increases and solving the corresponding integer minimization problem becomes involved.
\end{rem}

\section{The computation of $\varepsilon(d,k)$}\label{DLP.sec.4}

According to the discussion in Section~\ref{DLP.sec.1}, a lower bound on $\varepsilon(d,k)$ can be obtained by considering all the sets of $d+1$ pairwise distinct points from $\{0,\ldots,k\}^d$ and for each such set $\mathcal{S}$, all the partitions of $\mathcal{S}$ into two subsets $\{p^0,\ldots,p^n\}$ and $\{q^0,\ldots,q^m\}$ where $n+m$ is equal to $d-1$. For each such partition, one can build a matrix $A$ and a vector $b$ according to (\ref{DLP.sec.1.eq.4}) and (\ref{DLP.sec.1.eq.4.5}). The smallest possible non-zero value of the right-hand side of (\ref{DLP.sec.1.eq.5}) over all the obtained pairs $(A,b)$ such that $A^tA$ is non-singular will then be a lower bound on $\varepsilon(d,k)$. However, this strategy requires to consider
$$
N=(2^{d+1}-2){(k+1)^d\choose{d+1}}
$$
pairs $(A,b)$. Note that while this number would decrease to at best
$$
\frac{2^{d+1}-2}{2^dd!}{(k+1)^d\choose{d+1}}
$$
if the enumeration could be performed up to the symmetries of the $d$-dimensional hypercube. However, these symmetries are not all easy to handle in practice as one still needs to enumerate all $N$ pairs $(A,b)$ just to check for them.

We adopt a different strategy in order to significantly decrease the search space without having to handle symmetries. The main idea is to do the enumeration coordinate-wise in order to build a list $\mathcal{L}$ of the possible rows for the pair $(A,b)$ for each $n$ and $m$ that sum to $d-1$ and such that $n\leq{m}$, and then building $(A,b)$ back by selecting $d$ pairwise different rows from $\mathcal{L}$. By a \emph{row of $(A,b)$}, we mean a vector $r$ from $\mathbb{R}^d$ whose first $d-1$ entries form a row of $A$ and whose last entry is the corresponding coordinate of $b$. Note that our requirement that the rows of $\mathcal{L}$ selected to build a given pair $(A,b)$ are pairwise distinct is without loss of generality. Indeed, if two of these rows would coincide, a pair of columns of $A^tA$ would be multiples of one another and that matrix would then be singular. We shall see that the size of $\mathcal{L}$ does not depend on $n$ or $m$. As a consequence, this alternative strategy only considers
$$
\biggl\lfloor\frac{d+1}{2}\biggr\rfloor{|\mathcal{L}|\choose{d}}
$$
pairs $(A,b)$. For each of these pairs such that $A^tA$ is non-singular, the right-hand side of (\ref{DLP.sec.1.eq.5}) is evaluated, and the smallest non-zero value obtained for this quantity over all the considered pairs $(A,b)$ is the desired lower bound on $\varepsilon(d,k)$. Note that the efficiency of this strategy depends on how large $\mathcal{L}$ is. Let us get into more details about how we build $\mathcal{L}$.

For a given pair of positive integers $n$ and $m$ that sum to $d-1$, we generate all the possible rows of the pair $(A,b)$ as
\begin{equation}\label{DLP.sec.2.eq.0}
\bigl(x_1-x_0,\ldots,x_n-x_0,y_1-y_0,\ldots,y_m-y_0,y_0-x_0\bigr)
\end{equation}
where $x$ is a point from $\{0,\ldots,k\}^n$, $y$ is a point from $\{0,\ldots,k\}^m$ such that $x$ and $y$ are not both equal to $0$, and $x_0$ and $y_0$ are two integers from  $\{0,\ldots,k\}$. The list obtained from this procedure contains at most $(k-1)^{d+1}$ rows. It can be reduced using the following property.

\begin{prop}\label{DLP.sec.2.prop.1}
Consider a $d\mathord{\times}(d-1)$ matrix $A$ with integer entries such that $A^tA$ is non-singular and a vector $b$ from $\mathbb{Z}^d$. If the pair $(\overline{A},\overline{b})$ is obtained by dividing each row of the pair $(A,b)$ by the greatest common divisor of its coordinates and by negating a subset of the resulting rows, then
\begin{equation}\label{DLP.sec.2.prop.1.eq.0}
\bigl\|A(A^tA)^{-1}A^tb-b\bigr\|\geq\bigl\|\overline{A}(\overline{A}^t\overline{A})^{-1}\overline{A}^t\overline{b}-\overline{b}\bigr\|
\end{equation}
and the two sides of this inequality are either both zero or both positive.
\end{prop}
\begin{proof}
Pick two non-negative integers $n$ and $m$ that sum to $d-1$. Denote $b$ by $q^0$. Likewise, denote by $p^1$ to $p^n$ the first $n$ columns of $A$ and consider the points $q^1$ to $q^m$ from $\mathbb{Z}^d$ such that $q^1-q^0$ to $q^m-q^0$ are the last $m$ columns of $A$. According to the construction described in Section~\ref{DLP.sec.1},
\begin{equation}\label{DLP.sec.2.prop.1.eq.1}
\bigl\|A(A^tA)^{-1}A^tb-b\bigr\|^2=\min_{\substack{\lambda\in\mathbb{R}^n\\\mu\in\mathbb{R}^m}}f_{P,Q}(\lambda,\mu)
\end{equation}
where $P$ is the convex hull of $p^0$ to $p^n$ and $Q$ that of $q^0$ to $q^m$. Further denote by $(\overline{A},\overline{b})$ the pair obtained by dividing each row of $(A,b)$ by the greatest common divisor of its coordinates and by negating a fixed (but otherwise arbitrary) subset of the resulting rows. As above,
\begin{equation}\label{DLP.sec.2.prop.1.eq.2}
\bigl\|\overline{A}(\overline{A}^t\overline{A})^{-1}\overline{A}^t\overline{b}-\overline{b}\bigr\|^2=\min_{\substack{\lambda\in\mathbb{R}^n\\\mu\in\mathbb{R}^m}}f_{\overline{P},\overline{Q}}(\lambda,\mu)
\end{equation}
where $\overline{P}$ and $\overline{Q}$ are the convex hulls of the points $\overline{p}^0$ to $\overline{p}^n$ and $\overline{q}^0$ to $\overline{q}^m$ that are extracted from $(\overline{A},\overline{b})$ just as the points $p^i$ and $q^i$ are extracted from $(A,b)$. By construction, for any pair $(\lambda,\mu)$ of vectors in $\mathbb{R}^n\!\mathord{\times}\mathbb{R}^m$,
\begin{multline}\label{DLP.sec.2.prop.1.eq.2.5}
\sum_{j=1}^d\biggl(-q^0_j+\sum_{i=1}^n\lambda_ip^i_j-\sum_{i=1}^m\mu_i\bigl(q^i_j-q^0_j\bigr)\biggr)^{\!2}\\
=\sum_{j=1}^dr_j^2\biggl(-\overline{q}^0_j+\sum_{i=1}^n\lambda_i\overline{p}^i_j-\sum_{i=1}^m\mu_i\bigl(\overline{q}^i_j-\overline{q}^0_j\bigr)\biggr)^{\!2}
\end{multline}
where $r_1$ to $r_d$ the greatest common divisors of the rows of $(A,b)$. Observe that according to (\ref{DLP.sec.1.eq.1.5}), the right-hand side of this equality is precisely $f_{P,Q}(\lambda,\mu)$. Since the numbers $r_1$ to $r_d$ are not less than $1$, its right-hand side is at least $f_{\overline{P},\overline{Q}}(\lambda,\mu)$ and it follows that, for every $(\lambda,\mu)$ in $\mathbb{R}^n\!\mathord{\times}\mathbb{R}^m$,
$$
f_{P,Q}(\lambda,\mu)\geq{f_{\overline{P},\overline{Q}}(\lambda,\mu)}\mbox{.}
$$

In turn, by (\ref{DLP.sec.2.prop.1.eq.1}) and (\ref{DLP.sec.2.prop.1.eq.2}), the desired inequality holds. It remains to show that if the right-hand side of (\ref{DLP.sec.2.prop.1.eq.0}) is equal to $0$, then so is its left-hand side.

Assume that the right-hand side of (\ref{DLP.sec.2.prop.1.eq.0}) is equal to $0$. In that case, there exists a pair $(\lambda,\mu)$ of vectors in $\mathbb{R}^n\!\mathord{\times}\mathbb{R}^m$ such that $f_{\overline{P},\overline{Q}}(\lambda,\mu)$ is equal to $0$. According to (\ref{DLP.sec.1.eq.1.5}), $f_{\overline{P},\overline{Q}}(\lambda,\mu)$ is the squared norm of a vector and since it is equal to $0$, all of the coordinates of that vector must be equal to $0$. In other words, for every integer $j$ satisfying $1\leq{j}\leq{d}$,
$$
-\overline{q}^0_j+\sum_{i=1}^n\lambda_i\overline{p}^i_j-\sum_{i=1}^m\mu_i\bigl(\overline{q}^i_j-\overline{q}^0_j\bigr)=0
$$
and it follows from (\ref{DLP.sec.1.eq.1.5}) and (\ref{DLP.sec.2.prop.1.eq.2.5}) that $f_{P,Q}(\lambda,\mu)$ must vanish. According to (\ref{DLP.sec.2.prop.1.eq.1}), the left-hand side of (\ref{DLP.sec.2.prop.1.eq.0}) is then equal to $0$, as desired.
\end{proof}

\begin{table}[b]
\begin{center}
\begin{tabular}{>{\centering}p{0.7cm}cccccccccccccc}
\multirow{2}{*}{$d$}&  \multicolumn{10}{c}{$k$}\\
\cline{2-11}
& $1$ & $2$ & $3$ & $4$ & $5$ & $6$ & $7$ & $8$ & $9$ & $10$\\
\hline
$3$ & $6$ & $24$ & $72$ & $144$ & $288$ & $432$ & $720$ & $1008$ & $1440$ & $1872$\\
$4$ & $14$ & $89$ & $359$ & $929$ & $2189$ & $4019$ & $7469$ & $11969$\\
$5$ & $30$ & $300$ & $1620$ & $5400$ & $15120$\\
$6$ & $62$ & $965$ & $6971$\\
$7$ & $126$ & $3024$ & 
\end{tabular}
\end{center}
\caption{The number of rows in $\mathcal{L}$ as a function of $d$ and $k$.}
\label{DLP.sec.2.tab.1}
\end{table}

By Proposition~\ref{DLP.sec.2.prop.1}, we can assume that when several of the generated rows are multiples of one another, only the one among them whose coordinates are relatively prime and whose first non-zero coordinate is positive is included in $\mathcal{L}$. It should be observed that, before $\mathcal{L}$ is reduced as we have just described, its size does not depend on $n$ or $m$. As announced, this property still holds once $\mathcal{L}$ has been reduced. Indeed, observe that the rows generated by (\ref{DLP.sec.2.eq.0}) with the same two points $x$ and $y$ and the same scalars $x_0$ and $y_0$ but with different values of $n$ and $m$ can be recovered from one another by adding $y_0-x_0$ to (or subtracting this quantity from) certain of their coordinates. Hence, all of these rows have the same greatest common divisor for their coordinates.

We report in Table~\ref{DLP.sec.2.tab.1} as a function of $d$ and $k$ the number of rows contained in $\mathcal{L}$ after this procedure has been carried out. Note that, when $k$ is equal to $1$, $d$ to $3$, $n$ to $1$, and $m$ to $1$, there are only $6$ rows in $\mathcal{L}$:
$$
\mathcal{L}=\Bigl\{(1,0,0),\,(0,1,0),\,(1,-1,0),\,(0,1,-1),\,(1,1,-1),\,(1,0,-1)\Bigr\}\mbox{.}
$$

Twenty pairs $(A,b)$ are generated from this list of rows. When $k$ is equal to $3$ and $d$ to $4$, the number of pairs $(A,b)$ that have to be considered with the approach outlined at the beginning of the section is
$$
(2^5-2){4^4\choose5}=264\,286\,471\,680
$$
which would shrink down to
$$
\frac{2^5-2}{2^44!}{4^4\choose5}=688\,246\,020
$$
if these pairs could be enumerated up to the symmetries of the hypercube. However, there is no easy way to test for these symmetries without reviewing all of the $264\,286\,471\,680$ pairs. With our approach, $|\mathcal{L}|$ is equal to $359$ in that case as shown in Table~\ref{DLP.sec.2.tab.1} and we only need to consider
$$
2{359\choose{4}}=1\,361\,176\,502
$$
pairs $(A,b)$ in order to compute our lower bound on $\varepsilon(4,3)$.

This strategy does not only provide a lower bound on $\varepsilon(d,k)$ but also pairs $(A,b)$ that achieve this lower bound. Keeping track of the points $x$ and $y$ and of the scalars $x_0$ and $y_0$ that are used to build each row in $\mathcal{L}$ according to (\ref{DLP.sec.2.eq.0}), one can recover two lattice $(d,k)$-polytopes $P$ and $Q$ such that the obtained lower bound on $\varepsilon(d,k)$ is precisely the distance of the affine hulls of $P$ and $Q$. If the distance between these affine hulls coincides with the distance between $P$ and $Q$, which can easily be checked from the pair $(A,b)$ according to Remark~\ref{DLP.sec.1.rem.1}, then this lower bound is sharp. Interestingly, using this observation, all the lower bounds on $\varepsilon(d,k)$ that we have obtained using the presented computational strategy have turned out to be the precise value of $\varepsilon(d,k)$.

The known values of $\varepsilon(d,k)$ are reported in Table~\ref{DLP.sec.2.tab.2} in the case when $d$ is at least $3$. Recall that $\varepsilon(2,k)$ is given by Theorem~\ref{DLP.sec.3.thm.1} for all $k$ greater than $1$. Moreover, $\varepsilon(2,1)$ is equal to $1/\sqrt{2}$. The values shown in bold characters in the table are the ones that were obtained using the strategy described in this section. The other values of $\varepsilon(d,k)$ given in the table were previously obtained in \cite{DezaOnnPokuttaPournin2024}. Let us now describe the polytopes $P$ and $Q$ that have been obtained in each case whose distance is precisely $\varepsilon(d,k)$.

If $k$ is equal to $2$, $4$, $5$, $6$, $7$, or $8$, then $\varepsilon(3,k)$ is achieved by the line segment $P$ with vertices $(0,0,0)$ and $(k-1,k,k)$ and the line segment $Q$ with vertices $(k,1,2)$ and $(0,k,k-1)$. These two line segments are shown in Figure~\ref{DLP.sec.2.fig.1} when $k$ is equal to $4$. Note that in general, their distance is
\begin{equation}\label{DLP.sec.2.eq.1}
d(P,Q)=\frac{1}{\sqrt{2(2k^2-4k+5)(2k^2-2k+1)}}
\end{equation}
and it is tempting to ask whether $\varepsilon(3,k)$ is equal to this value for every integer $k$ greater than $8$. In any case, note that the right-hand side of (\ref{DLP.sec.2.eq.1}) provides an upper bound on $\varepsilon(3,k)$ that decreases like $1/(2\sqrt{2}k^2)$ as $k$ goes to infinity. In the remaining two cases, when $k$ is equal to either $1$ or $3$, the line segments that achieve $\varepsilon(3,k)$ do not follow the pattern we have just described. Indeed, $\varepsilon(3,1)$ is the distance between a diagonal $P$ of the cube $[0,1]^3$ and a diagonal $Q$ of one of its square faces such that $P$ and $Q$ are disjoint (see also \cite{DezaOnnPokuttaPournin2024}) while $\varepsilon(3,3)$ is the distance between the line segment with vertices $(0,0,0)$ and $(2,3,3)$ and the line segment with vertices $(3,2,0)$ and $(0,1,2)$.

\begin{figure}[t]
\begin{centering}
\includegraphics[scale=1]{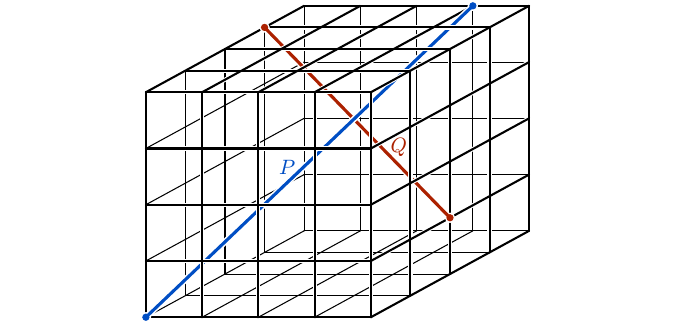}
\caption{A pair of kissing lattice $(3,4)$-polytopes.}\label{DLP.sec.2.fig.1}
\end{centering}
\end{figure}

In the $4$-dimensional case, the values of $\varepsilon(4,k)$ reported in Table~\ref{DLP.sec.2.tab.2} are always achieved by a line segment $P$ and a triangle $Q$ as follows. When $k$ is equal to $1$, the vertices of $P$ are $(0,0,0,0)$ and $(1,1,1,1)$ while the vertices of $Q$ are $(1,0,0,0)$, $(0,1,1,0)$, and $(0,1,0,1)$. When $k$ is greater than $1$, the vertices of $P$ are $(0,0,0,0)$ and $(1,2,1,2)$ and those of $Q$ are $(2,2,1,0)$, $(0,1,0,2)$, and $(0,0,2,1)$. When $k$ is equal to $3$, the vertices of $P$ are $(0,0,1,0)$ and $(2,3,3,3)$ and the vertices of $Q$ are $(3,0,3,2)$, $(0,2,0,3)$, and $(0,3,3,0)$.

The unique value of $\varepsilon(5,k)$ reported in Table~\ref{DLP.sec.2.tab.2} is the one when $k$ is equal to $1$. This value is the distance between the diagonal of the hypercube $[0,1]^5$ incident to the origin of $\mathbb{R}^5$ and the tetrahedron with vertices $(1,1,0,0,0)$, $(0,1,0,1,1)$, $(0,0,1,0,1)$, and $(0,0,1,1,0)$. Finally, $\varepsilon(6,k)$ is equal to the distance between the diagonal of the hypercube $[0,1]^6$ incident to the origin of $\mathbb{R}^6$ and the $5$\nobreakdash-dimensional simplex with vertices $(1,0,1,1,0,0)$, $(1,0,0,0,1,1)$, $(0,1,1,0,1,1)$, $(0,1,0,1,0,1)$, and $(0,1,0,1,1,0)$.

\begin{rem}
It is noteworthy that, while $\varepsilon(2,k)$ can be guessed from the pairs of kissing polytopes obtained for the first few values of $k$, this is not the case for $\varepsilon(d,k)$ when $k$ is fixed. Even when $k$ is equal to $1$, the pairs of kissing polytopes known for the first few values of $d$ do not exhibit a clear pattern.
\end{rem}

\bibliography{SmallKissingPolytopes}
\bibliographystyle{ijmart}

\end{document}